\newcommand{\abs}[1]{\lvert#1\rvert}
\documentclass[12pt]{article}

\usepackage[psamsfonts]{amssymb}
\usepackage{amsfonts,amsmath}
\usepackage{amsthm}
\usepackage{graphicx}

\newcounter{minutes}
\divide\time by 60
\newcounter{hours}
\multiply\time by 60 \addtocounter{minutes}{-\time}

\newtheorem{theorem}{Theorem}
\newtheorem{lemma}{Lemma}

\newtheorem{proposition}{Proposition}[section]
\newtheorem{definition}{Definition}[section]
\newtheorem{remark}{Remark}[section]

\usepackage{amsmath, amssymb, amscd, amsxtra, txfonts}

\title{Invertible harmonic mappings of unit disk onto Dini smooth Jordan domains}
\author{David Kalaj}
\date{}
\begin{document}

\maketitle

\def\thefootnote{}
\footnotetext{
\texttt{\tiny File:~\jobname .tex,
          printed: \number\year-\number\month-\number\day,
          \thehours.\ifnum\theminutes<10{0}\fi\theminutes}
}
\makeatletter\def\thefootnote{\@arabic\c@footnote}\makeatother

\begin{abstract}
In this paper we extend Rado-Choquet-Kneser theorem for the mappings with weak homeomorphic Lipschitz
boundary data and Dini's smooth boundary but  without restriction on the convexity of image domain, provided that the Jacobian satisfies a certain boundary condition.
The proof is based on a recent extension of Rado-Choquet-Kneser
theorem by Alessandrini and Nesi \cite{ale} and it is used the
approximation principle.

\end{abstract}
\maketitle 
\section{Introduction}

Harmonic mappings in the plane are univalent complex-valued harmonic
functions of a complex variable. Conformal mappings are a special
case where the real and imaginary parts are conjugate harmonic
functions, satisfying the Cauchy-Riemann equations. Harmonic
mappings were studied classically by differential geometers because
they provide isothermal (or conformal) coordinates for minimal
surfaces. More recently they have been actively investigated by
complex analysts as generalizations of univalent analytic functions,
or conformal mappings. For the background to this theory we refer to
the book of Duren \cite{dure}. If $w$ is a univalent complex-valued
harmonic functions, then by Lewy's theorem (see \cite{l}), $w$ has a
non-vanishing Jacobian and consequently, according to the inverse
mapping theorem, $w$ is a diffeomorphism.
Moreover, if $w$ is a harmonic mapping of the unit disk $\mathbf U$
onto a convex Jordan domain $\Omega$, mapping the boundary $\mathbf
T=\partial \mathbf U$ onto $\partial \Omega$ homeomorphically, then
$w$ is a diffeomorphism. This is celebrated theorem of Rado, Knesser
and Choquet (\cite{knes}). This theorem has been extended in various
directions (see for example \cite{jj}, \cite{an}, \cite{sy} and
\cite{sf}). One of the recent extensions is the following
proposition, due to Nesi and Alessandrini \cite{ale}, which is one of the main
tools in proving our main result.


\begin{proposition}\label{ales} Let $F: \mathbf T\to
\gamma\subset \mathbf C$ be an orientation preserving diffeomorphism
of class $C^1$ onto a simple closed curve. Let $D$ be a bounded
domain such that $\partial D =\gamma$. Let $w=P[F]\in
C^1(\overline{\mathbf U};\mathbf C).$ The mapping $w$ is a
diffeomorphism of $\mathbf U$ onto $D$ if and only if
\begin{equation}\label{use} J_w(\zeta)
> 0\text{ everywhere on}\  \mathbf T,\end{equation} where $J_w(\zeta)
:=\lim_{r\to 1-}J_w(r\zeta)$, and $J_w(r\zeta)$ is the Jacobian of
$w$ at $r\zeta$.
\end{proposition}
%
%
\section{Statement of the main result and preliminaries}

In this paper we generalize Rado-Kneser-Choquet theorem as follows.

\begin{theorem}[The main result]\label{ma} Let $F: \mathbf T\to
\gamma\subset \mathbf C$ be an orientation preserving Lipschitz weak
homeomorphism of the unit circle $\mathbf T$ onto a Dini's smooth Jordan curve and let $w=P[F]$ be its Poisson extension in the unit disk. Let $D$ be a bounded domain such that $\partial
D =\gamma$. Then there is a continuous function $T(\zeta)$ on  $\mathbf T$  such that $$J_w(\zeta)= |\partial_t F(\zeta)| T(\zeta),\ \ \  \text{for a.e.   }\zeta=e^{it}\in \mathbf{T}.$$ If
\begin{equation}\label{use12}
T(\zeta)>0 \ \ \text{  in }   \ \ \ \mathbf{T}
,\end{equation}  then the mapping $w=P[F]$ is a
diffeomorphism of $\mathbf U$ onto $D$.
\end{theorem}

 In order to compare this statement with Kneser's
Theorem, it is worth noticing that, when $D$ is convex, then by
Remark~\ref{rep} the condition \eqref{use12} is automatically
satisfied. In a paper of the author in \cite{studia} the author proved Theorem~\ref{ma} under a stronger condition that $\gamma\in C^{1,\alpha}$ for some $\alpha>0$ and by using a different approach.

Note that we do not have any restriction on convexity of image
domain in Theorem~\ref{ma} which is  proved in section~3.

\subsection*{A conjecture} Motivated by Theorem~\ref{ma} we state the following conjecture. Let $F: \mathbf T\to \gamma\subset \mathbf C$ be a homeomorphism, where $\gamma$ is a rectifiable Jordan curve. Let $D$ be the
bounded domain such that $\partial D =\gamma$. The mapping $w=P[F]$
is a diffeomorphism of $\mathbf U$ onto $D$ if and only if
\begin{equation}\label{use15}
\mathrm{ess\,inf}\{J_w(\zeta):\zeta \in\mathbf{T}\}\ge 0.\end{equation}

Before proving results we overview the involved concepts and
make the main definitions concerning this paper.
\subsection{Smooth Jordan curves and their parameterizations} Suppose  that $\gamma$ is a
rectifiable curve in the complex plane. Denote  by $l$ the length of
$\gamma$ and let $g:[0,l]\mapsto \gamma$ be the arc length
parameterization of $\gamma$, i.e. the parameterization satisfying
the condition:
$| g'(s)|=1 \text{ for all $s\in [0,l]$}.$

Let
\begin{equation}\label{ker}K(s,t)=\text{Re}\,[\overline{(g(t)-g(s))}\cdot i
g'(s)]\end{equation} be a function defined on $[0,l]\times[0,l]$. By
$K(s\pm l, t\pm l)=K(s,t)$ we extend it on $\mathbf R\times \mathbf
R$. Note that $ig'(s)$ is the inner unit normal vector of $\gamma$
at $g(s)$ and therefore, if $\gamma$ is convex then
\begin{equation}\label{convexkernel}K(s,t)\ge 0\text{ for every $s$
and $t$}.\end{equation} Suppose now that $F:\mathbf R\mapsto \gamma$
is an arbitrary $2\pi$ periodic Lipschitz function such that
$F|_{[0,2\pi)}:[0,2\pi)\mapsto \gamma$ is an orientation preserving
bijective function. { Then there exists an increasing continuous
function $f:[0,2\pi]\mapsto [0,l]$ such that
\begin{equation}\label{fgs}F(\tau)=g(f(\tau)).\end{equation}}
In the remainder of this paper we will identify $[0,2\pi)$ with the
unit circle $\mathbf T$,  and $F(s)$ with $ F(e^{is})$. In view of
the previous convention we have
$$F'(\tau)=g'(f(\tau))\cdot f'(\tau),$$ and therefore
$$|F'(\tau)|=|g'(f(\tau))|\cdot |f'(\tau)|=f'(\tau).$$
Along with the function $K$ we will also consider the function $K_F$
defined by
$$K_F(t,\tau)=\text{Re}\,[\overline{(F(t)-F(\tau))}\cdot
iF'(\tau)].$$ It is easy to see that
\begin{equation}\label{kg}K_F(t,\tau)=f'(\tau)K(f(t), f(\tau)).
\end{equation}

\begin{definition} Let $l=|\gamma|$.
We will say that a surjective function $F=g\circ f:\mathbf T\to
\gamma$ is a weak homeomorphism, if $f:[0,2\pi]\to[0,l]$ is a
nondecreasing surjective function.
\end{definition}

\begin{definition}
Let $f:[a,b]\to \mathbf C$ be a continuous function. The modulus of
continuity of $f$ is
   $$\omega(t)= \omega_f(t) = \sup_{|x-y|\le t} |f(x)-f(y)|. \,$$
The function $f$ is called Dini-continuous if

\begin{equation}\label{didi}\int_{0^+} \frac{\omega_f(t)}{t}\,dt < \infty.\end{equation} Here
$\int_{0^+}:=\int_{0}^k$ for some positive constant $k$.  A $C^1$
Jordan curve $\gamma$ with the length $l=|\gamma|$, is said to be
Dini smooth if $g'$ is Dini continuous. We say that  $\gamma$ is of class $C^{1,\mu}$, $0<\mu\le 1$,
if $g$  is of class   $C^1$    and
$$\sup_{t\neq s}\frac{|g'(t)-g'(s)|}{|t-s|^{\mu}}<\infty.$$ Observe that every smooth
$C^{1,\mu}$ Jordan curve is Dini smooth.
\end{definition}

\begin{lemma} \cite{studia}
If $\gamma$ is Dini smooth, and $\omega$ is modulus of continuity of
$g'$, then
\begin{equation}\label{oh}|K(s,t)|\le
\int_0^{\min\{|s-t|,l-|s-t|\}}\omega(\tau)d\tau.\end{equation}
\end{lemma}

A function $\varphi:A\to B$ is called  $(\ell,\mathcal{L)}$
bi-Lipschitz, where $0<\ell<\mathcal{L}<\infty$, if $\ell|x-y|\le
|\varphi(x)-\varphi(y)|\le\mathcal{L} |x-y|$ for $x,y\in A$.

 \subsection{Harmonic mappings} A mapping
$w$ is called \emph{harmonic} in a region $D$ if $w=u+iv$ where $u$
and $v$ are harmonic functions in $D$. If $D$ is simply-connected,
then there are two analytic functions $g$ and $h$ defined on $D$
such that
$$w=g+\overline h.$$
Let $$P(r,t)=\frac{1-r^2}{2\pi (1-2r\cos t+r^2)}$$ denote the
Poisson kernel. The Poisson integral of $F\in L^1(\mathbf T)$ is a
harmonic function given by
\begin{equation}\label{e:POISSON}
w(z)=P[F](z)=\int_0^{2\pi}P(r,t-\tau)F(e^{it})dt,
\end{equation}
where $z=re^{i\tau}\in \mathbf U$.

\section{The proof of the main theorem}

The aim of this chapter is to prove Theorem~\ref{ma}. As in \cite{studia}, we will
construct a suitable sequence $w_n$ of univalent harmonic mappings,
converging almost uniformly to $w=P[F]$.  In order to do so, we will
mollify the boundary function $F$, by a sequence of diffeomorphism
$F_n$ and take the Poisson extension $w_n=P[F_n]$. We will show
that, under the condition of Theorem~\ref{ma} for large $n$, $w_n$
satisfies the conditions of theorem of Alessandrini and Nesi. By a
result of Hengartner and Schober \cite{hsc}, the limit function $w$
of a locally uniformly convergent sequence of univalent harmonic
mappings $w_n$ is univalent, providing that $F$ is a surjective
mapping.

We begin by the following lemmata.
\begin{lemma}\label{dri}(See e.g. \cite{studia}).
If $\varphi:\mathbf R\to \mathbf R$ is a ${L}-$ Lipschitz weak homeomorphism,
such that $\varphi(x+a) = \varphi(x) + b$ for some $a$ and $b$ and
every $x$, then there exist a sequence of ${L}-$ Lipschitz
diffeomorphisms) $\varphi_n:\mathbf R\to \mathbf R$ such that
$\varphi_n$ converges uniformly to $\varphi$, and $\varphi_n(x+a) =
\varphi_n(x)+b$.
\end{lemma}

\begin{lemma}\cite{studia}
If $\omega:[0,l]\to \mathbf R_+$ is a bounded function satisfying
\eqref{didi}, then for every constant $a$, $\omega(ax)$
satisfies \eqref{didi}. Next for every $0<y\le l$ there hold the
following formula:
\begin{equation}\label{goodd}\int_{0+}^y\frac{1}{x^2}\int_0^x\omega(a t) dt dx =
\int_{0+}^y\left[\frac{\omega(ax)}{x}-\frac{\omega(ax)}{y}\right]dx.\end{equation}
\end{lemma}
\begin{lemma}\label{form}\cite{studia}
If $w=P[F]$ is a harmonic mapping, such that $F=g\circ f$, where $g$
and $f$ are as mentioned earlier, is a Lipschitz weak homeomorphism
from the unit circle onto a Dini smooth Jordan curve, then for
almost every $\tau\in [0,2\pi]$
$$J_w(e^{i\tau}) :=\lim_{r\to 1} J_w(re^{i\tau})$$ and there hold the
formula

\begin{equation}\label{jacfor}\begin{split}J_w(e^{i\tau})&= f'(\tau)\int_0^{2\pi}
\frac{\mathrm{Re}\,[\overline{(g(f(t))-g(f(\tau)))}\cdot i
g'(f(\tau))]}{2\sin^2\frac{t-\tau}{2}}dt.\end{split}\end{equation}
\end{lemma}

For a Lipschitz non-decreasing function $f$ and an arc-length
parametrization $g$ of the Dini's smooth curve $\gamma$ we define
the operator $T$ as follows
\begin{equation}\label{Ff}T[f](\tau)=\int_0^{2\pi}
\frac{\mathrm{Re}\,[\overline{(g(f(t))-g(f(\tau)))}\cdot i
g'(f(\tau)))]}{2\sin^2\frac{t-\tau}{2}}dt, \tau\in
[0,2\pi].\end{equation} According to Lemma~\ref{form}, this integral
converges.
\begin{remark}\label{rep}
Notice that, if $\gamma$ is a convex Jordan curve then
$\mathrm{Re}\,[\overline{(g(f(t))-g(f(\tau)))}\cdot i
g'(f(\tau))]\ge 0$, and therefore $T[f]>0$. In the next proof, we
will show that $T[f]$ is continuous if $f$ is a Lipschitz weak homeomorphism  and under the integral condition $T[f]>0$ the harmonic
extension of a bi-Lipschitz mapping is a diffeomorphism regardless
of the condition of convexity.

\end{remark}

\begin{proof}[Proof of Theorem~\ref{ma}]
Assume for simplicity that $|\gamma|= 2\pi$. The general case
follows by normalization. Let $g:[0,2\pi]\to \gamma$ be na arc length
parametrization of $\gamma$. Then $F(e^{it})=g(f(t))$, where
$f:\mathbf R\to \mathbf R$ is a $L-$Lipschitz weak homeomorphism such
that $f(t+2\pi) = f(t) + 2\pi$.  By using Lemma~\ref{dri}, we can choose a family
of $L-$Lipschitz diffeomorphisms $f_n$ converging uniformly to $f$. Then
\begin{equation}\label{under}T[f_n]=\int_0^{2\pi}
\frac{\mathrm{Re}\,[\overline{(g(f_n(t))-g(f_n(\tau)))}\cdot i
g'(f_n(\tau)))]}{2\sin^2\frac{t-\tau}{2}}dt, \tau\in
[0,2\pi].\end{equation} We are going to show that $T[f_n]$ converges
uniformly to $T[f]$. In order to do this we apply Arzela-Ascoli
theorem. Let $\beta:[0,2\pi]\to R$ be a function such that $g'(s)=e^{i\beta(s)}$. 

As in \cite{studia} we obtain $$T[f](\tau)=\int_{-\pi}^{\pi}f'(t+\tau)\cdot
\sin[\beta(f(t+\tau))-\beta(f(\tau))]\cot \frac t2
\frac{dt}{2\pi}.$$ Then \[\begin{split}|T[f_n]|&\le
\frac{1}{\pi}\|f_n'\|\int_{0}^{\pi}\omega_\beta(\|f_n'\|
t)\cot \frac t2 dt
\\&\le \frac{1}{\pi}\|f'\|\int_{0}^{\pi}\omega_\beta(\|f'\| t)\cot \frac t2 dt
= C(f,\gamma)<\infty.\end{split}\] We prove now that $T[f_n]$ is
equicontinuous family of functions. In the contrary, there is $x_\circ\in[0,2\pi]$ and $\epsilon_\circ>0$ and there is a sequence of non-decreasing natural numbers  $n_k$ and a sequence of real numbers $\tau_k$ tending to zero such that for every $k\in\mathbf{N}$ we have \begin{equation}\label{epsilon}|T[f_{n_k}](x_\circ+\tau_k)-T[f_{n_k}](x_\circ)|\ge \epsilon_\circ.\end{equation} Assume without loos of generality that $x_\circ = 0$.
Use the notation $$\mathcal{K}_k(t,\tau) = \mathrm{Re}\,[\overline{(g(f_{n_k}(t))-g(f_{n_k}(\tau)))}\cdot i
g'(f_{n_k}(\tau))].$$
On the other hand by \eqref{oh} and \eqref{kg}, for $$\sigma_k =
\min\{|f_{n_k}(t+\tau_k)-f_{n_k}(\tau_k)|, l-|f_{n_k}(t+\tau_k)-f_{n_k}(\tau_k)|\}$$ we obtain
$$\left|\mathcal{K}_k(t+\tau_k,\tau_k)\right|\le
\|f'_{n_k}\|\int_0^{\sigma_k}\omega(u)du,$$ where $\omega$ is the
modulus of continuity of $g'$. Since $\|f'_{n_k}\|\le \|f'\|$ we have
\begin{equation}\label{eee}\begin{split}\frac{\abs{\mathcal{K}_k(t+\tau_k,\tau_k)}}{2\sin^2\frac{t}{2}}&\le \frac{\|f'\|}{2\sin^2\frac{t}{2}}\int_0^{\sigma_k}\omega(u)du\\&\le\frac{\sigma_k}{t}\frac{\|f'\|\pi}{4
t^2}\int_0^{t}\omega\left(\frac{\sigma_k}{t}u\right)du\\&\le
\frac{\pi\|f'\|^2}{4t^2}\int_0^{t}\omega(\|f'\|u)du:=Q(t).\end{split}\end{equation}
Thus $Q(t)$ is a dominant for the expressions
$$\frac{\abs{\mathcal{K}_k(t+\tau_k,\tau_k)}}{2\sin^2\frac{t}{2}}.$$ Having
in mind the equation \eqref{goodd}, we obtain
\[\begin{split}\int_{-\pi}^{\pi} \abs{Q(t)}dt&\le
\frac{2\pi\|f'\|^2}{2}\int_0^{\pi}\frac{1}{t^2}\int_0^{t}\omega(\|f'\|u)du
\\&={\pi\|f'\|^2}\int_0^{\pi}\left(\frac{\omega(\|f'\|u)}{u}-\frac{\omega(\|f'\|u)}{\pi}
\right)du \\&<\infty.\end{split}\]
According to the
Lebesgue Dominated Convergence Theorem, taking the limit when $k\to \infty$ under the
integral sign in the integral \eqref{under} we obtain that $\displaystyle\lim_{k\to \infty} T_{f_{n_k}}(\tau_k)=T_{f_{n_\circ}}(0)$ if $n_k$ is a stationary sequence and $\displaystyle\lim_{k\to \infty} T_{f_{n_k}}(\tau_k)=T_{f}(0)$ in the other case. Similarly $\displaystyle\lim_{k\to \infty} T_{f_{n_k}}(0)=T_{f_{n_\circ}}(0)$ if $n_k$ is a stationary sequence and $\displaystyle\lim_{k\to \infty} T_{f_{n_k}}(0)=T_{f}(0)$ in the other case and this contradicts \eqref{epsilon}.

This implies that the family $\{T[f_n]\}$ is
equicontinuous.  By Arzela-Ascoli theorem it follows that
$$\lim_{n\to \infty} \|T[f_n] - T[f]\|=0.$$ Thus $T[f]$ is
continuous and for $\epsilon=\min_{t}T[f](t)$, there is $n_0$ such that $\|T[f_n] - T[f]\|\le \epsilon/2$ and therefore $T[f_n](t)\ge \epsilon/2$ for $n\ge n_0$ and $t\in[0,2\pi]$.
\\
Moreover, since $f_n$ is a diffeomorphism, for $n$ sufficiently
large there holds the following inequality
$$J_{w_n}(e^{i\tau})=f_n'(\tau)T[f_n](\tau)>0, \tau\in
[0,2\pi].$$ Since $f_n\in C^{\infty}$, it follows that
$w_n = P[F_n]\in C^1(\overline{\mathbf U}).$ Therefore all the
conditions of Proposition~\ref{ales} are satisfied. This means that
$w_n$ is a harmonic diffeomorphism of the unit disk onto the domain
$D$.

Since, by a result of Hengartner and Schober \cite{hsc}, the limit
function $w$ of a locally uniformly convergent sequence of univalent
harmonic mappings $w_n$ on $\mathbf U$ is either univalent on
$\mathbf U$, is a constant, or its image lies on a straight-line, we
obtain that $w=P[F]$ is univalent, because $F$ is a surjective function of $\mathbf{T}$ onto $\gamma$. The proof is completed.
\end{proof}

\bigskip

\noindent D. Kalaj -  Faculty of natural sciences and mathematics, University of Montenegro, Montenegro,  D\v zord\v za Va\v singtona  b.b. 81000, Podgorica, Montenegro, {davidk@t-com.me}

\end{document}